\documentclass[10pt]{amsart}
\usepackage{a4} 
\usepackage{amssymb}
\usepackage{array}
\usepackage{booktabs}
\usepackage{multirow}
\usepackage{hhline}

\newtheorem{thm}{Theorem}

\newtheorem{lem}[thm]{Lemma}
\newtheorem{prop}[thm]{Proposition}
\newtheorem{rem}[thm]{Remark}

\numberwithin{thm}{section}
\numberwithin{equation}{section}
\newcommand{\Real}{\mathbb R}

\newcommand{\norm}[1]{\left\Vert#1\right\Vert}
\newcommand{\abs}[1]{\left\vert#1\right\vert}

\newcommand{\eps}{\varepsilon}

\newcommand{\la}{\langle}
\newcommand{\ra}{\rangle}

\newcommand{\A}{\mathcal{A}}

\newcommand{\F}{\mathcal{F}}

\newcommand{\Hi}{\mathcal{H}}
\newcommand{\I}{\mathcal{I}}

\newcommand{\M}{\mathcal{M}}
\newcommand{\n}{\mathbb{N}}
\newcommand{\N}{\mathcal{N}}

\newcommand{\U}{\mathcal{U}}

\newcommand{\Com}{\mathbb{C}}

\newcommand{\Om}{\Omega}

\begin{document}
\title[$q$-Hypercontractivity]
{Hypercontractivity on the $q$-Araki-Woods algebras}
\author{Hun Hee Lee}
\author{\'Eric Ricard*}

\address{Hun Hee Lee :
Department of Mathematics, Chungbuk National University,
410 Sungbong-Ro, Heungduk-Gu, Cheongju 361-763, Korea}
\email{hhlee@chungbuk.ac.kr}
\address{ \'Eric Ricard:
Laboratoire de Math\'ematique, Universit\'e de Franche-Comt\'e,
16 Route de Gray, 25030 Besan{\c c}on, France}
\email{eric.ricard@univ-fcomte.fr}

\thanks{*Research partially supported by ANR grant 06-BLAN-0015}

\begin{abstract}
Extending a work of Carlen and Lieb, Biane has obtained the optimal
hypercontractivity of the $q$-Ornstein-Uhlenbeck semigroup on the
$q$-deformation of the free group algebra. In this note, we look for
an extension of this result to the type III situation, that is for the
$q$-Araki-Woods algebras. We show that hypercontractivity from $L^p$
to $L^2$ can occur if and only if the generator of the deformation is
bounded.
\end{abstract}

\maketitle

\section{Introduction}

In \cite{Ne} Nelson proved the following famous hypercontractivity result for the classical Ornstein-Uhlenbeck semigroup $P^1_t$
acting on $L^p(\Real^n, d\gamma)$, where $d\gamma$ is the $n$-dimensional gaussian measure on $\Real^n$.
	
	\begin{thm} {\bf (Nelson, 1973)}\\
		For $1< p < r <\infty$ we have
			$$\norm{P^1_t}_{L^p \rightarrow L^r} \le 1\;\; \text{if and only if}\;\; e^{-2t} \le \frac{p-1}{r-1}.$$
	\end{thm}

	Since then, there have been several analogous results in the context
of non-commutative probability.  A fermionic counterpart of the
Nelson's result has been clarified by Carlen/Lieb in \cite{CL93}, and
including the fermionic case Biane proved in \cite{Bi97} the following
hypercontractivity result for the $q$-Ornstein-Uhlenbeck semigroup
$P^q_t$ acting on $L^p(\Gamma_q, \tau_q)$, where $\Gamma_q$ is the von
Neumann algebra generated by $q$-gaussians by Bo\.zejko and Speicher
(\cite{BS91, BS94}) and $\tau_q$ is the vacuum state.
	
	\begin{thm} {\bf (Biane, 1997)}\\
		Let $-1\le q < 1$. For $1< p < r <\infty$ we have
			$$\norm{P^q_t}_{L^p \rightarrow L^r} \le 1\;\; \text{if and only if}\;\; e^{-2t} \le \frac{p-1}{r-1}.$$
	\end{thm}
The above result has been further extended to the case of gaussians satisfying more general commutation relations by Krolak (\cite{Kr05}),
and the holomorphic version has been proved by Kemp (\cite{Kem05}).
	
Biane's generalized result concerns only about von Neumann algebras
with a normal tracial state.  Thus, it is natural to be interested in
their non-tracial relatives, namely, $q$-Araki-Woods algebras
$\Gamma_q(H_\Real,(U_t))$ (\cite{Hi} or see section
\ref{sec-q-AW-alg}), which is a generalization of Araki-Woods factors
(\cite{AW69}) depending on the deformation group of orthogonal transformations $(U_t)$ on
some real Hilbert space $H_\Real$. As usual, we denote by $A$ the
(unbounded) generator of $(U_t)$ on the complexification of $H_\Real$.
Fortunately, we were able to prove the following hypercontractivity
result for the $q$-Ornstein-Uhlenbeck semigroup $P^{q,(U_t)}_t$
(simply $P^q_t$ again) acting on $L^p(\Gamma_q(H_\Real,(U_t)),
\tau_q)$, where $\tau_q$ is the vacuum state.

	\begin{thm}\label{thm-general}
		Let $-1\le q < 1$. For $1< p <2$ we have
			$$\norm{P^q_t}_{L^p \rightarrow L^2} \le 1\;\; \text{if}\;\; e^{-2t} \le C\norm{A}^{1-\frac{2}{p}}(p-1),$$
		where $C$ is a universal constant.
		
		If $A$ is unbounded, then ${P^q_t}$ is not bounded from $L^p$ to $L^2$.
	\end{thm}

Following the idea of Carlen/Lieb and Biane we will use the baby Fock
model and Speicher's central limit procedure as key ingredients for
the proof.  Although we were unable to determine ``the optimal time"
for the contractivity, the constant in the above shows the
same optimal order $p-1$ as in the tracial case.

This paper is organized as follows.  In section \ref{sec-q-AW-alg} we
will review about $q$-generalized gaussians, $q$-Araki-Woods algebras,
and $q$-Ornstein-Uhlenbeck semigroup.  We will also briefly review of
an extension procedure of a map on non-tracial von Neumann algebras
to the associated non-commutative $L^p$ spaces.
In section \ref{sec-babyFock} we will introduce the {\it
twisted baby Fock} model by Nou (\cite{N06}), and we will prove the
first main result, hypercontractivity of $\eps$-Ornstein-Uhlenbeck
semigroup, which is an analog of $q$-Ornstein-Uhlenbeck semigroup in
the baby Fock model.  In section \ref{sec-CentralLimit} we will use
Speicher's central limit procedure and hypercontractivity of
$\eps$-Ornstein-Uhlenbeck semigroup to prove hypercontractivity of the
$q$-Ornstein-Uhlenbeck semigroup.  Finally in section \ref{sec-1dim}
we will use 1-dimensional estimate to get the converse direction of
the main result from the previous section.
	
\section{$q$-Araki-Woods algebras and $q$-Ornstein-Uhlenbeck semigroup}\label{sec-q-AW-alg}

We begin with the most general definition of $q$-Araki-Woods algebra
(see \cite{AW69, Hi, Sh97}).  The construction starts with a (separable) real Hilbert space
$H_\Real$ and $(U_t)_{t\in \Real}$ a group of orthogonal transformations
on $H_\Real$. Let $H_\Com=H_\Real+i H_\Real$ be the complexification of
$H_\Real$. The group $(U_t)$ extends to a group of unitaries on
$H_\Com$. By Banach-Stone's theorem, there is a self-adjoint
(unbounded) operator $A$ so that $U_t=A^{it}$. One then defines a new
scalar product on $H_\Com$ by
$$\la x,y\ra_U=\la \frac {2A}{1+A} x,y\ra_{H_\Com}.$$ Its completion
is then denoted by $\Hi$.

We consider the
operator of symmetrization $P_n$ on $\Hi^{\otimes n}$ defined by $$P_0
\Om = \Om,$$
    $$P_n (f_1\otimes \cdots \otimes f_n) = \sum_{\pi \in S_n}q^{i(\pi)}f_{\pi(1)}\otimes \cdots \otimes f_{\pi(n)},$$
where $S_n$ denotes the symmetric group of permutations of $n$ elements and
    $$i(\pi) = \# \{(i,j) | 1\leq i,j \leq n, \pi(i) > \pi(j)\}$$ is
    the number of inversions of $\pi \in S_n$.

Now we define the {\it $q$-inner product} $\left\langle \cdot, \cdot
\right\rangle_q$ on $\Hi^{\otimes n}$ by
    $$\left\langle \xi, \eta \right\rangle_q =\delta_{n,m}\left\langle
    \xi, P_n \eta \right\rangle \;\, \text{for}\;\, \xi\in
    \Hi^{\otimes n}, \eta\in\Hi^{\otimes m},$$ where $\left\langle
    \cdot, \cdot \right\rangle$ is the inner product in $\Hi^{\otimes n}$.
We denote by $\Hi^{\otimes_q n}$ the resulting Hilbert space.  Since
    $P_n$'s are strictly positive for $-1<q<1$ (\cite{BS91}),
    $\left\langle \cdot, \cdot \right\rangle_q$ is actually an inner
    product.  
Then one can associate a $q$-Fock space
$\F_q(\Hi)$.
    $$\F_q(\Hi) = \mathbb{C}\Om \oplus \bigoplus_{n\geq 1}\Hi^{\otimes_q n},$$ where $\Om$ is a unit vector called vacuum.

When $q=0$ we recover the classical  full Fock space over $\Hi$.
In the extreme cases $q = \pm 1$, $\F_1(\Hi)$ and $\F_{-1}(\Hi)$ refer to Bosonic and Fermionic Fock spaces, respectively.
In the whole paper we are interested in $-1\le q < 1$.

For $h\in H_\Real$, we can define a generalized $q$-semi-circular random
variable on $\F_q(\Hi)$ by
$$ s(h)=\ell(h)+\ell^*(h),$$ where $\ell_q(h)$ is the left creation
operator by $h\in \Hi$ and $\ell^*_q(h)$ is the adjoint of
$\ell_q(h)$.

By definition $\Gamma_q(H_\Real,(U_t)) = \{s(h) : h\in H_\Real\}''$.
The vacuum state $\tau_q$ defined by $\tau_q(\cdot) =
\left\langle \,\cdot\, \Om, \Om \right\rangle_q$ is a normal faithful
state on $\Gamma_q(H_\Real,(U_t))$, and $(\Gamma_q(H_\Real,(U_t)),\tau_q,\Om)$ is in GNS position.
$\Gamma_q(H_\Real,(U_t))$ is known to be a type $II_1$ algebra if and
only if $(U_t)$ is trivial, in general, it is a type III von Neumann algebra,
whose modular theory relative to $\tau_q$ is well understood (\cite{Sh97, Hi}). 

Now we consider the $q$-Ornstein-Uhlenbeck semigroup $P^{q,(U_t)}_t$ (simply $P^q_t$).
Since $\Om$ is a separating vector for $\Gamma_q(H_\Real,(U_t))$ we can use the second quantization (\cite[Proposition 1.1]{Hi}, \cite[Theorem 2.11]{BKS})
to get the semigroup $P^q_t : \Gamma_q \rightarrow \Gamma_q$ given by
	$$P^q_t(X)\Om = \F(e^{-t}id_\Hi)X\Om,$$
where $\F(e^{-t}id_\Hi) : \F_q(\Hi) \rightarrow \F_q(\Hi)$ defined by
	$$\F(e^{-t}id_\Hi)|_{\Hi^{\otimes n}} = e^{-nt}id_{\Hi^{\otimes n}},\;\; n\ge 0.$$
Note that second quantizations commute with the modular group of $\tau_q$.
For the later use we record the properties of the semigroup as follows.
	\begin{prop}
	$P^q_t$, $t\ge 0$ is a completely positive, normal, $\tau_q$-preserving contraction that commutes with the modular group of $\tau_q$.
	\end{prop}

Recall that  $A$ is said to be almost periodic when it has an orthonormal basis of
eigenvectors. In this situation, there is a more tractable model for
$\Gamma_q(H_\Real,(U_t))$ that we will use (see Section 2.2 in \cite{N06}). 

Assume that we are given a sequence $\mu = (\mu_i)_{i\ge 1} \subseteq [1,\infty)$.
The construction starts with a separable complex Hilbert space $\Hi$
equipped with an orthonormal basis $(e_{\pm k})_{k\geq 1}$. We denote
as above by $\F_q(\Hi)$ the associated $q$-Fock space.
We define {\it $q$-(generalized) gaussian variables} (or {\it $q$-generalized
circular variables)} by
    $$g_{q,i} = \mu^{-1}_i\ell_q(e_i) + \mu_i \ell^*_q(e_{-i}).$$
Let $\Gamma_q^{\mu}$ ($-1\leq q < 1$), the von Neumann algebra generated by $\{g_{q,k}\}_{k \geq 1}$
and $\tau_q$ still denotes the vacuum state.

If $A$ is almost periodic and $(\mu_i^4)_{i\ge 1} \subseteq [1,\infty)$
is the sequence of eigenvalues of $A$ that are bigger than 1, then there is a
spatial isomorphism
	$$(\Gamma_q(H_\Real,(U_t)),\tau_q,\Om) \cong (\Gamma_q^\mu,\tau_q,\Om).$$
Of course, one has $\norm{A}=\sup \mu_i^4$.

In this model, the $q$-Ornstein-Uhlenbeck semigroup $P^{q,\mu}_t$ (simply $P^q_t$) is also obtained by the second quantization procedure with the same formula.

All the properties we will be looking at are stable under taking ultraproducts,
so that the discretization procedure in Section 6.1 of \cite{N06} allows
us to work only in the almost periodic situation.

\medskip

In this paper we will often need an extension of a map defined on the algebra level to the $L^p$ setting.
Let $\M$, $\N$ be von Neumann algebras with distinguished normal faithful states $\varphi$ and $\psi$,
respectively. Then, one can define the Haagerup's $L^p$ space $L^p(\M, \varphi)$, $1\le p < \infty$.
For the details about $L^p(\M, \varphi)$ (simply $L^p(\M)$) we refer \cite{JX03, PX03}.
Let $T: \M \rightarrow \N$ be a completely positive contraction.
We say $T$ is {\it state-preserving} if $\psi\circ T=\varphi$ and
it intertwines the modular groups: $\sigma^\psi_t\circ T=T\circ \sigma^\varphi_t$ for all $t\in \Real$.
Let tr$_\M$ be the trace on $L^1(\M)$ and $D_\varphi$ be the density operator associated to
$\varphi$. Similarly, we consider tr$_\N$ and $D_\psi$.  Then $\M
D^{\frac{1}{p}}_\varphi$ is norm-dense in $L^p(\M)$, and for $x\in \M$
the elements $xD^{\frac{1}{p}}_\varphi \in L^p(\M)\; (1\le p <
\infty)$ are identified in the sense of complex interpolation.  Now we
consider an extension of $T$ to $L^p$ setting given by $T^p : \M
D^{\frac{1}{p}}_\varphi \rightarrow \N D^{\frac{1}{p}}_\psi, \;\;
xD^{\frac{1}{p}}_\varphi \mapsto (Tx)D^{\frac{1}{p}}_\psi.$ It is
well-known that $T^p$ can be extended to a contraction on $L^p(\M)$ (\cite[Lemma 2.2]{JX03}).
For a later use we record the extension procedure as follows.
	\begin{prop}\label{prop-Lp-ext}
	Let $\M$, $\N$ be von Neumann algebras with distinguished
	normal faithful states $\varphi$ and $\psi$, respectively, and
	$T: \M \rightarrow \N$ be a state-preserving completely positive
	contraction. Then
		\begin{equation}\label{eq-Lp-ext}
		T^p : \M D^{\frac{1}{p}}_\varphi \rightarrow \N D^{\frac{1}{p}}_\psi,
		\;\; xD^{\frac{1}{p}}_\varphi \mapsto (Tx)D^{\frac{1}{p}}_\psi.
		\end{equation}
	extends to a contraction on $L^p(\M)$.
	
	In particular, if $T$ is an onto isometry, then so is $T^p$ between $L^p(\M)$ and $L^p(\N)$.
	\end{prop}

Inspired by the above we can consider the following further extension of $T$.
Let $1\le p, r < \infty$.
	\begin{equation}\label{eq-Lp-Lr-ext}
		T^{p,r} : \M D^{\frac{1}{p}}_\varphi \rightarrow \N D^{\frac{1}{r}}_\psi \subseteq L_r(\N),
		\;\; xD^{\frac{1}{p}}_\varphi \mapsto (Tx)D^{\frac{1}{r}}_\psi.
	\end{equation}
In general, there is no guarantee that $T^{p,r}$ can be extended to a bounded map from $L^p(\M)$ into $L_r(\N)$ when $r>p$.

If we apply the above extension \eqref{eq-Lp-ext} to the $q$-Ornstein-Uhlenbeck semigroup $P^q_t$,
we obtain a contractive semigroup $P^{q,p}_t$ on $L^p(\Gamma_q)$.
	$$P^{q,p}_t : L^p(\Gamma_q) \rightarrow L^p(\Gamma_q),\; xD^{\frac{1}{p}}_q \mapsto P^q_t(x)D^{\frac{1}{p}}_q,$$
where $D_q$ is the density operator associated to the vacuum state $\tau_q$.

Now the question is for which $1 < p < r < \infty$ and for which
$t>0$, can we extend $P^q_t$ to a contraction from $L^p(\Gamma_q)$
into $L^r(\Gamma_q)$? More precisely, when does the map
	\begin{equation}\label{eq-OU-Lp-Lr}
		P^{q,p,r}_t : \Gamma_q D^{\frac{1}{p}}_q \rightarrow \Gamma_q D^{\frac{1}{r}}_q, \;\;
		xD^{\frac{1}{p}}_q \mapsto (P^q_t x)D^{\frac{1}{r}}_q
	\end{equation}
can be extended to a contraction from $L^p(\Gamma_q)$ into $L^r(\Gamma_q)$?
Here comes a partial answer to this question, which is one of our main results.
Note that we will simply denote $P^{q,p,r}_t$ again by $P^q_t$.

	\begin{thm}\label{thm-main-if}
	Let $1<p<2$ and $P^q_t$ is the map in \eqref{eq-OU-Lp-Lr}. Then, we have 
		$$\norm{P^q_t}_{L^p \rightarrow L^2} \le 1\;\; \text{if}\;\;
		e^{-2t} \le C\alpha_\mu^{4-\frac{8}{p}}(p-1)$$
	for some universal constant $C>0$, where $\alpha_\mu = \sup_{n \ge 1} \mu_n$.
	\end{thm}

We close this section with some precise results on the modular theory for $\Gamma_q$ and $\tau_q$ that we need.
The modular group $\sigma^{\tau_q}_t$ with respect to $\tau_q$ satisfies the following.
    $$\sigma^{\tau_q}_t(g_{q,k}) = \mu_k^{4it}g_{q,k}, \;\; k\ge 1.$$
Thus, $g_{q,k}$ is an analytic element satisfying
    \begin{equation}\label{gq-in-Lp}
    D^{\frac{1}{2p}}_q g_{q,k}  = \mu_k^{\frac{2}{p}} g_{q,k} D^{\frac{1}{2p}}_q, \;\; k\ge 1.
    \end{equation}

\section{Twisted baby Fock and Hypercontractivity of $\eps$-Ornstein-Uhlenbeck semigroup}\label{sec-babyFock}

\subsection{The baby Fock model}

We will briefly describe the twisted baby Fock introduced by A. Nou
(\cite{N06}).  Let $I = \{\pm 1, \pm 2, \cdots, \pm n\}$ be a fixed
index set and $\eps : I\times I\rightarrow \{\pm 1\}$ be a ``choice of
sign" function satisfying
	$$\eps(i,j) = \eps(j,i),\; \eps(i,i)=-1,\; \eps(i,j) = \eps(\abs{i}, \abs{j}),\; \forall i,j\in I.$$
Now, we consider the unital algebra $\A(I,\eps)$ with generators $(x_i)_{i\in I}$ satisfying
	$$x_ix_j - \eps(i,j)x_jx_i = 2\delta_{i,j},\;\; i,j\in I.$$
In particular, we have $x^2_i = 1$, $i\in I$, where $1$ refers to the unit of the algebra.
$\A(I,\eps)$ can be endowed with the involution given by $x^*_i=x_i$.
We will use the following notations for the elements in $\A(I,\eps)$.
	$$x_{\emptyset} := 1 \;\;\text{and}\;\; x_A := x_{i_1}\cdots x_{i_k},\;
	A = \{i_1< \cdots < i_k\} \subseteq I.$$
Then, $\{x_A : A\subseteq I\}$ is a basis for $A(\I,\eps)$.
Let $\phi^\eps : \A(I,\eps) \rightarrow \mathbb{C}$ be the tracial state given by
	$$\phi^\eps(x_A) = \delta_{A,\emptyset}.$$
$\phi^\eps$ give rise to a natural inner product on $\A(I,\eps)$ as follows.
	$$\la x, y \ra := \phi^\eps(y^*x),\;\, x,y\in \A(I,\eps).$$
Let
	$$H = L^2(\A(I,\eps),\phi^\eps)$$
be the corresponding $L^2$-space, then clearly $\{x_A : A\subseteq I\}$ is an orthonormal basis for $H$.

Now we consider left creations $\beta^*_i$ and left annihilations $\beta_i$ in $B(H)$, $i\in I$ in this context.
	$$\beta^*_i(x_A) = 
	\left\{\begin{array}{ll}x_ix_A & \text{if}\; i\notin A\\0 & \text{if}\; i\in A \end{array}\right.,
	\;\;\beta_i(x_A) = 
	\left\{\begin{array}{ll}x_ix_A & \text{if}\; i\in A\\0 & \text{if}\; i\notin A \end{array}\right.,\;\;
	i\in I,\; A\subseteq I.$$
Then using the same parameter $\mu = (\mu_i)_{i\ge 1}$ we define our generalized gaussians
	$$\gamma_i := \mu^{-1}_i\beta^*_i + \mu_i\beta_{-i},\;\; 1\le i\le n.$$
The following relations are known to be satisfied by $\gamma_i$'s (\cite[Lemma 5.2]{N06}).
\begin{equation}\label{rela}\left\{\begin{array}{ll}
	\gamma_i\gamma_j - \eps(i,j)\gamma_j\gamma_i = 0 & i\neq j \in I\\
	\gamma^*_i\gamma_j - \eps(i,j)\gamma_j\gamma^*_i = 0 & i\neq j \in I\\
	\gamma^2_i = (\gamma^*_i)^2 = 0 & i\in I\\
	\gamma^*_i\gamma_i + \gamma_i\gamma^*_i = (\mu^2_i + \mu^{-2}_i)id & i\in I
	\end{array}\right.\end{equation}
Let $\Gamma_{\la 1,\cdots,n \ra}$ be the von Neumann algebra generated
by $\{\gamma_i\}^n_{i=1}$ in $B(H)$ while $\Gamma_{\la n \ra}$ refers
to the von Neumann algebra generated by $\gamma_n$.  It is also known
that (\cite[Lemma 5.2]{N06}) $1$ is a cyclic and separating vector for
$\Gamma_{\la 1,\cdots,n \ra}$, and the above is the faithful GNS
representation of $(\Gamma_{\la 1,\cdots,n \ra}, \tau_n^\eps)$, where
$\tau_n^\eps$ is the vacuum state on $\Gamma_{\la 1,\cdots,n \ra}$
given by $\tau_n^\eps(\cdot) = \la\, \cdot\, 1, 1\ra$. With this
definition, $\Gamma_{\la 1,\cdots,k \ra}$ is not a subalgebra of
$\Gamma_{\la 1,\cdots,n \ra}$, but by the above facts we can 
indeed identify it with the subalgebra generated by $\gamma_1,...,\gamma_k$ in 
 $\Gamma_{\la 1,\cdots,n \ra}$ and $\tau^\eps_k$ is then the restriction of 
$\tau^\eps_n$. We may sometimes  simply write $\tau^\eps$ without any reference to $n$.  Finally, we remark that with the
classical identification $x\mapsto x1$ between $\Gamma_{\la 1,\cdots,n
\ra}$ and $H$, $\tau^\eps$ corresponds to $\phi^\eps$.

We collect some results about $\gamma_i$'s which we need in the sequel.

	\begin{prop}\label{prop-gamma}We have:
		\begin{enumerate}
		\item For all $n\geq 1$, and $i\leq n$, $\sigma^{\tau^\eps}_t(\gamma_i)=\mu^{4it} \gamma_i$ and $\tau^\eps(\gamma_i^*\gamma_i)=\mu_i^{-2}$.	
		\item Let $D_{\la 1,\cdots,n \ra}$ and $D_{\la 1,\cdots,n-1 \ra}$ be the densities of $\tau^\eps$
			restricted to $\Gamma_{\la 1,\cdots,n \ra}$ and $\Gamma_{\la 1,\cdots,n-1 \ra}$, respectively.
			Then we have a natural isometric embedding
				$$L^p(\Gamma_{\la 1,\cdots,n-1 \ra}) \hookrightarrow L^p(\Gamma_{\la 1,\cdots,n \ra}),
				\;\; x D_{\la 1,\cdots,n-1 \ra}^{\frac{1}{p}} \mapsto x D_{\la 1,\cdots,n \ra}^{\frac{1}{p}}.$$
			A similar statement holds for $L^p(\Gamma_{\la n \ra})$.
			\item For any $a \in \Gamma_{\la n \ra}$ and $b,\,c\in \Gamma_{\la 1,\cdots,n-1 \ra}$ we have
				$$\tau^\eps (cab) = \tau^\eps(a) \tau^\eps(cb).$$
			\item $\gamma^*_n\gamma_n$ commutes with $\Gamma_{\la 1,\cdots,n-1 \ra}$ and $D_{\la 1,\cdots,n \ra}$.
		\end{enumerate}
	\end{prop}
\begin{proof}
The computations for the first point can be found in \cite[Proposition 5.3]{N06}. Thus, the natural inclusion of $\Gamma_{\la 1,\cdots,n-1
  \ra}$ to $\Gamma_{\la 1,\cdots,n \ra}$ is state preserving and one
can use Proposition \ref{prop-Lp-ext}. For the third point, it suffices to do it with $c=1$ using the modular theory, but then it is \cite[Lemma 5.4]{N06}.
Finally for $i<n$
$$\gamma^*_n\gamma_n\gamma_i=\gamma^*_n(\eps(n,i)\gamma_i\gamma_n)=\eps(n,i)^2
\gamma_i\gamma^*_n\gamma_n=\gamma_i\gamma^*_n\gamma_n.$$ And
$\sigma^{\tau^\eps}_t(\gamma^*_n\gamma_n)=\sigma^{\tau^\eps}_t(\gamma_n)^*
\sigma^{\tau^\eps}_t(\gamma_n)=\gamma^*_n\gamma_n$, so
$\gamma^*_n\gamma_n$ is in the centralizer of $\tau^\eps$ which exactly means that it commutes with $D_{\la 1,\cdots,n \ra}$.
\end{proof}

\begin{rem}
{\rm
Note that we are dealing with finite dimensional algebras, so we know
that the density $D_{\la 1,\cdots,n \ra}$ belongs to $\Gamma_{\la
1,\cdots,n \ra}$. Moreover all associated $L^p$-spaces can be identified with 
a $p$-Schatten classes provided that we fix a faithful trace on $\Gamma_{\la
1,\cdots,n \ra}$ (instead of looking at the trace on the dual space $\Gamma_{\la 1,\cdots,n \ra}^*$).
}
\end{rem}

The main estimates relies on the position of $\Gamma_{\la 1, \cdots, n-1\ra}$
inside $\Gamma_{\la 1, \cdots, n\ra}$, which will be clarified in the following Proposition.
Now we set
	\begin{equation}\label{y_n}
		y_i = \gamma^*_i\gamma_i - \mu^{-2}_i id,\; 1\le i \le n
	\end{equation}
In this paper we will consider the basis $\{1, \gamma_i, \gamma^*_i,
y_i\}$ of $\Gamma_{\la i \ra}$ which is more suitable than
$\{\gamma_i\gamma^*_i, \gamma_i, \gamma^*_i, \gamma^*_i\gamma_i\}$,
$1\le i \le n$, as their corresponding $L^2$ vectors are orthogonal
and have length 0, 1, 1 and 2, respectively.  Then, for a fixed $1\le
i \le n$, any element $X \in \Gamma_{\la 1, \cdots, n\ra}$ can be {\it
uniquely} expressed in the form
	\begin{equation}\label{unique-rep}
	X = a + \gamma_i b + \gamma^*_i c + y_i d
	\end{equation}
for some $a,b,c,d \in \Gamma_{\la 1, \cdots, n \ra \backslash i}$, the von Neumann algebra generated by $\{\gamma_k : 1\le k\neq i \le n \}$.

	\begin{prop}\label{form}
	There are a $*$-isomorphism $\Phi : \Gamma_{\la 1,\cdots,n \ra}\to \mathbb M_2\otimes \Gamma_{\la 1,...,n-1\ra}$
	and a unitary $u_n\in \Gamma_{\la 1,...,n-1\ra}$ satisfying the following.
		\begin{enumerate}
			\item For $a \in \Gamma_{\la 1,...,n-1\ra}$ we have
			$\Phi(a)=\left[\begin{array}{cc} a & 0\\ 0& a \end{array}\right].$
			\item $\Phi(\gamma_n)=(\mu_n^2+\mu_n^{-2})^{\frac{1}{2}}\left[\begin{array}{cc} 0 & 0\\ u_n& 0 \end{array}\right]$.
			\item $\tau^\eps_n=(\psi\otimes\tau^{\eps}_{n-1})\Phi$, where 
			$\psi \left(\left[\begin{array}{cc} x_{11} & x_{12}\\ x_{21}& x_{22}
			\end{array}\right]\right)=\lambda x_{11} +(1-\lambda)x_{22}$ with $\lambda=\frac 1{1+\mu_n^4}$.
		\end{enumerate}
	\end{prop}
\begin{proof}

The relations \ref{rela} give us a $*$-isomorphism $\sigma : \Gamma_{\la n\ra} \to \mathbb M_2$
with $\sigma(\gamma_n)=(\mu_n^2+\mu_n^{-2})^{1/2}e_{21}$.
Let $C_n$ be the relative commutant of $\Gamma_{\la n\ra}$ in $\Gamma_{\la 1,\cdots,n \ra}$.
Then, the multiplication map $\psi: \Gamma_{\la n\ra}\otimes C_n\to \Gamma_{\la 1,\cdots,n \ra}$ is a $*$-homorphism.
Actually, $\psi$ is a $*$-isomorphism. Indeed, we can easily check that $\psi$ is an onto map.
For example, we have
	$$\psi\big( (\eps(1,n)\gamma^*_n\gamma_n + \gamma_n\gamma^*_n)
	\otimes \gamma_1 (\eps(1,n)\gamma^*_n\gamma_n +
	\gamma_n\gamma^*_n) \big) = (\mu^2_n +
	\mu^{-2}_n)^2\gamma_1.$$ The fact that $\gamma_1
	(\eps(1,n)\gamma^*_n\gamma_n + \gamma_n\gamma^*_n) \in C_n$ is
	a straightforward calculation.  Moreover, $\psi$ must be a 1-1
	map, since dim$C_n = 2^{2n-2}$, which can be checked by a
	simple induction and the following observation.  Let $X = a +
	\gamma_1 b + \gamma^*_1 c + y_1 d$, $a,b,c,d \in \Gamma_{\la
	2, \cdots, n \ra}$.  Then by the uniqueness of the expression
	\eqref{unique-rep} we have
	$$X \in C_n\;\; \Leftrightarrow\;\; a,d\in \Gamma'_{\la n
	\ra}\cap \Gamma_{\la 2,\cdots,n \ra} \;\; \text{and}\;\;
	\left\{\begin{array}{cc}\eps(1,n)\gamma_n b = b\gamma_n, &
	\eps(1,n)\gamma^*_n b = b\gamma^*_n \\ \eps(1,n)\gamma_n c =
	c\gamma_n, & \eps(1,n)\gamma^*_n c =
	c\gamma^*_n\end{array}\right..$$ Now we consider another
	$*$-isomorphism $\pi = (\sigma \otimes Id)\circ\psi^{-1} :
	\Gamma_{\la 1,\cdots,n \ra}\to \mathbb M_2\otimes C_n$. Since
	$\gamma_n^*\gamma_n$ commutes with $\Gamma_{\la 1,\cdots,n-1
	\ra}$ there are $*$-homomorphisms $\pi_i : \Gamma_{\la
	1,\cdots,n-1 \ra} \to e_{ii}\otimes C_n,\; a\mapsto
	(e_{ii}\otimes 1)\pi(a)(e_{ii}\otimes 1)$, $i=1,2$.
	Actually, $\pi_i$'s are $*$-isomorphisms since dim$
	\Gamma_{\la 1,\cdots,n-1 \ra} = 2^{2n-2} = \text{dim}C_n $
	(\cite[Lemma 5.2]{N06}) and they are injective. Indeed, we
	have
	\begin{align*}
		\pi^{-1}\big((e_{ii}\otimes id)\pi(a)(e_{ii}\otimes id)\big)
		& = (\mu_n^2+\mu_n^{-2})^{-2} \gamma^*_n\gamma_n a \gamma^*_n\gamma_n\\
		& = (\mu_n^2+\mu_n^{-2})^{-1} \gamma^*_n\gamma_n a.
	\end{align*}
Then, the injectivity comes from the uniqueness of the expression \eqref{unique-rep}.

Now by identifying $e_{ii}\otimes C_n$ and $C_n$ we get two $*$-isomorphisms
	$$\rho_1, \rho_2 : \Gamma_{\la 1,\cdots,n-1 \ra} \to C_n\;\; \text{such that}\;\;
	\pi(a) = \left[\begin{array}{cc} \rho_1(a) & 0\\ 0& \rho_2(a) \end{array}\right].$$
The above $*$-isomorphisms enable us to conclude that $\Gamma_{\la 1,\cdots,k \ra} \cong \mathbb M_{2^k}$, $k\ge 1$ by a simple induction.
Thus, any automorphism on $\Gamma_{\la 1,\cdots,n-1 \ra}$ is inner, so that there is a unitary $u_n \in \Gamma_{\la 1,\cdots,n-1 \ra}$
such that $\rho^{-1}_1\circ \rho_2(a) = u^*_n a u_n$, $a\in \Gamma_{\la 1,\cdots,n-1 \ra}$.
Finally we define $\Phi : \Gamma_{\la 1,\cdots,n \ra}\to \mathbb M_2\otimes \Gamma_{\la 1,...,n-1\ra}$ by
	$$\Phi(x) = \left[\begin{array}{cc} 1 & 0\\ 0& u^*_n \end{array}\right]
	\big[(I\otimes\rho^{-1}_1)\circ \pi\big](x) \left[\begin{array}{cc} 1 & 0\\ 0& u_n \end{array}\right].$$

Then, the first two assertions easily follow from the construction of $\Phi$ and $u_n$.
The formula for $\tau_n^\eps$ is a consequence of Proposition \ref{prop-gamma} (3) and the definition of $\Phi$.
For example, we have 
$\tau^\eps_n(\gamma_nb) = 0 =
\psi\otimes \tau^\eps_{n-1}\Big( \left[\begin{array}{cc} 0 & 0\\ (\mu_n^2+\mu_n^{-2})^{\frac{1}{2}}u_n b& 0 \end{array}\right] \Big)$
for $b\in \Gamma_{\la 1,...,n-1\ra}$, and
$\tau^\eps_n(y_n d) = 0  =
\psi\otimes \tau^\eps_{n-1}\Big( \left[\begin{array}{cc} \mu^2_n d & 0\\ 0 & -\mu^{-2}_n d \end{array}\right] \Big)$
for $d\in \Gamma_{\la 1,...,n-1\ra}$.
\end{proof}

Now we consider the number operator $N_\eps$ on $H$ given by $N_\eps =
\sum_{i\in I} \beta^*_i\beta_i$. Then, for any $A \subseteq I$ we have
$N_\eps x_A = \abs{A}x_A$. Since $1 \in H$ is separating and cyclic for 
$\Gamma_{\la 1,\cdots,n \ra}$ we define the $\eps$-Ornstein-Uhlenbeck semigroup $P^\eps_t : \Gamma_{\la 1,\cdots,n
\ra} \rightarrow \Gamma_{\la 1,\cdots,n \ra}$ by
	\begin{equation}\label{eps-OU-semigp}
	P^\eps_t(X)1 = e^{-tN_\eps}(X1),\;\; X\in \Gamma_{\la 1,\cdots,n \ra}.
	\end{equation}
To make this definition more explicit, any element in $\Gamma_{\la
1,\cdots,n \ra}$ can be written as a linear combination of products
$w_1...w_k$ where $w_i\in \{id,\gamma_i,\gamma_i^*,y_i\}$, $1\le i \le k$. The number
operator counts $0$ for $id$, $1$ for $\gamma_i$, $\gamma_i^*$ and 2
for $y_i$; for instance
$P^\eps_t(y_4\gamma_2^*\gamma_1)= e^{-4t}y_4\gamma_2^*\gamma_1$. This can be
checked by a straightforward induction.

In comparison to the $q$-Fock space setting, it not easy to see that this
defines a completely positive semigroup. There is no general second
quantization in the baby Fock model. Nevertheless, such a procedure
exists for some diagonal contractions. To do so, define similarly the 
$i$-number operator $N_i$ on $H$ and $T^t_i$ on $\Gamma_{\la 1,\cdots,n \ra}$ by
	$$N_i = \beta^*_i\beta_i+\beta^*_{-i}\beta_{-i},\; 1\le i \le n\;\;\text{and}\;\;
	T^t_i(X)1 = e^{-tN_i}(X1).$$
It counts only the letter $i$  as explained above.

\begin{prop}
For any $t\geq 0$, the operators $T^t_i$ $(1\le i \le n)$ are completely positive and state
preserving on $\Gamma_{\la 1,\cdots,n \ra}$, and so is $P^\eps_t$.
\end{prop}
\begin{proof}
The second assertion follows from the first one as $P_t^\eps=T_1^t \cdots T_n^t$.

For simplicity we only check the case $i=n$.  Let $a,b,c,d \in \Gamma_{\la 1,\cdots,n-1 \ra}$, we have
$T_n^t(a)=a$, $T_n^t(\gamma_n b)=e^{-t}\gamma_n b$, $T_n^t(\gamma^*_n c)=e^{-t}\gamma^*_n c$ and
$T_n^t(y_n d)=e^{-2t}y_n d$.  We use Proposition \ref{form} to transfer $T_n^t$
to $\widetilde T_n^t = \Phi \circ T_n^t \circ \Phi^{-1}$ on $\mathbb M_2\otimes \Gamma_{\la 1,\cdots,n \ra}$.
From the formula for $\Phi(\gamma_n)$ it follows that
$\widetilde T^t_n \left[\begin{array}{cc} 0 & b\\ 0& 0 \end{array}\right]
=e^{-t}\left[\begin{array}{cc} 0 & b\\ 0& 0 \end{array}\right]$
and since
$\Phi(a+y_nd)=\left[\begin{array}{cc} a+\mu_n^2d & 0\\ 0& a-\mu_n^{-2}d \end{array}\right]$
for $a,d \in \Gamma_{\la 1,\cdots,n-1 \ra}$ we have
	$$\widetilde T^t_n \left[\begin{array}{cc} a+\mu_n^2d & 0\\ 0& a-\mu_n^{-2}d \end{array}\right]
	=e^{-t}\left[\begin{array}{cc} a+\mu_n^2 e^{-2t}d & 0\\ 0& a-\mu_n^{-2}e^{-2t}d \end{array}\right].$$

Thus, we obtain that $\widetilde T_n^t = T\otimes Id$,
where $T(e_{12})=e^{-t} e_{12}$, $T(e_{21})=e^{-t} e_{21}$, $T(1)=1$ and
$T(\mu_n^2e_{11}-\mu_n^{-2}e_{22})=e^{-2t}(\mu_n^2e_{11}-\mu_n^{-2}e_{22})$.
We get with $\lambda=\frac 1{1+\mu_n^4}$
\begin{eqnarray*}
T(e_{11})&=& \lambda (1+e^{-2t}\mu_n^4) e_{11}+\lambda
(1-e^{-2t})e_{22} \\ T(e_{22})&=&(1-\lambda)(1-e^{-2t})
e_{11}+(1-\lambda)(1-e^{-2t}\mu_n^{-4})e_{22}\end{eqnarray*}
The Choi's matrice $C=(T(e_{i,j}))_{i,j}$ associated to $T$ is 
$$\left[\begin{array}{cccc}
\lambda (1+e^{-2t}\mu_n^4) & 0 & 0& e^{-t} \\
0 & \lambda (1-e^{-2t}) & 0 & 0 \\
0 & 0 & (1-\lambda)(1-e^{-2t}) & 0 \\
e^{-t} & 0 & 0& (1-\lambda)(1+e^{-2t}\mu_n^{-4})
\end{array}\right]$$
Since $\mu^4_n=\frac {1-\lambda}\lambda$ and 
\begin{eqnarray*}
 \lambda (1-\lambda) (1+e^{-2t}\mu_n^4)(1+e^{-2t}\mu_n^{-4})-e^{-2t}
&=& \lambda (1-\lambda)(1-e^{-2t})^2\geq0,
\end{eqnarray*}
$C$ is positive and $T$ is completely positive.

The state preserving property follows from Proposition \ref{prop-gamma}.
\end{proof}

\subsection{Main estimates}

We start with the main statement.
	\begin{thm}\label{thm-eps}
	Let $1<p\le 2$.
		$$\norm{P^\eps_t}_{L^p \rightarrow L^2} \le 1\;\;
		\text{if}\;\; e^{-2t} \le C\alpha_\mu^{4-\frac{8}{p}}(p-1)$$
	for some universal constant $C>0$, where $\alpha_\mu = \sup_{n \ge 1} \mu_n$.
	\end{thm}

Before proceeding to the proof, we collect some lemmas.
The first and the most crucial one is an asymmetric version of optimal convexity inequality (\cite{BCL, CL93}).
	\begin{lem}\label{lem-modified-CL}
	Let $1<p\le 2$, $\mu\geq1$ and $\lambda=\frac 1{1+\mu^4}$. For any $A, B\in \mathbb M_n$, $n\ge 1$ we have
		$$\Big(\lambda \norm{A+\mu^2 B}^p_p+(1-\lambda)\norm{A-\mu^{-2}B}^p_p\Big)^{\frac{2}{p}}
		\ge \norm{A}^2_p + C(p-1)\norm{B}^2_p$$
	where $C = C(\mu) = \frac{1}{3}\mu^{-4}$ for $1<p \le \frac{3}{4}$
	and $C = \frac{1}{3}\mu^{8-\frac{16}{p}}$ for $\frac{3}{4}<p\le 2$.
	\end{lem}
\begin{proof}
The above inequality is nothing but the contractivity of
a fixed linear map from an $L^p$ space to a $L^p$-valued $\ell^2$ space.
By a careful examination of the adjoint map we can observe that
the above inequality is equivalent to the following.
Let $\frac{1}{p}+\frac{1}{q} = 1$. Then for any $X,Y\in \mathbb M_n$ we have
	\begin{align}\label{dual-ineq}
	\Big(\lambda \norm{X+Y}^q_q+(1-\lambda)\norm{X-\frac{\lambda}{1-\lambda}Y}^q_q\Big)^{\frac{2}{q}}
	& \le \norm{X}^2_q + \frac{q-1}{\mu^4C}\norm{Y}^2_q.
	\end{align}
Since we care less about the best constant we will use the following standard argument.
Let $C_q$ be the best constant such that \eqref{dual-ineq} is true
if we replace $\frac{q-1}{\mu^4C}$ by $C_q$.
Then we have
	\begin{align*}
	\lefteqn{\lambda \norm{X+Y}^{2q}_{2q}+(1-\lambda)\norm{X-\frac{\lambda}{1-\lambda}Y}^{2q}_{2q}}\\
	& = \lambda \norm{\abs{X}^2+\abs{Y}^2+X^*Y+Y^*X}^q_q\\
	&\;\;\;\; + (1-\lambda) \norm{\abs{X}^2 + \Big(\frac{\lambda}{1-\lambda}\Big)^2\abs{Y}^2
	-\frac{\lambda}{1-\lambda}(X^*Y+Y^*X)}^q_q\\
	& \le \lambda \norm{\abs{X}^2+\abs{Y}^2+X^*Y+Y^*X}^q_q\\
	&\;\;\;\; + (1-\lambda) \norm{\abs{X}^2 + \abs{Y}^2-\frac{\lambda}{1-\lambda}(X^*Y+Y^*X)}^q_q\\
	& \le \Big(\norm{\abs{X}^2+\abs{Y}^2}^2_q + C_q \norm{X^*Y+Y^*X}^2_q \Big)^{\frac{q}{2}}\\
	& \le \Big([\norm{X}^2_{2q} + \norm{Y}^2_{2q}]^2
	+ 4C_q \norm{X}^2_{2q}\norm{Y}^2_{2q} \Big)^{\frac{q}{2}}\\
	& \le \Big(\norm{X}^2_{2q} + (2C_q+1)\norm{Y}^2_{2q} \Big)^{\frac{2q}{2}}.
	\end{align*}
The first inequality is by monotony of the $L^p$ norm on positive
elements as $\frac{\lambda}{1-\lambda}\le 1$.

Thus, we can conclude that $C_{2q}\le 2C_q+1$.
Since we have $C_2 = \mu^{-4}\leq 1$, a standard interpolation argument leads us to
$C_q \le (q-1)^{1-\theta}(2q-1)^\theta$, where $2^n\le q<2^{n+1}$ and
$\frac{1}{q} = \frac{1-\theta}{2^n} + \frac{\theta}{2^{n+1}}$.
Thus, we simply get
	$$C_q \le 3(q-1),$$
which implies that the original inequality is true for
	$$C = \frac{1}{3}\mu^{-4}.$$
When $n=1$, i.e. $2\le q<4$ we can get a sharper estimate.
Since $C_2 = \mu^{-4}$	and $C_4 \le 2\mu^{-4}+1$,
for $\frac{1}{q} = \frac{1-\theta}{2} + \frac{\theta}{4}$ we have
	$$C_q \le (\mu^{-4})^{1-\theta}(2\mu^{-4}+1)^\theta \le 3\mu^{4-\frac{16}{q}},$$
which implies that the original inequality is true for
	$$C = \frac{1}{3}\mu^{8-\frac{16}{p}}.$$
\end{proof}
	
\begin{lem}\label{lem-a+yn}
For any $a,\, b\in \Gamma_{\la 1,\cdots,n-1 \ra}$, with $\lambda=\frac 1{\mu_n^4+1}$, we have 
\begin{align}\label{eq-estimate-I}
	\norm{(a+y_nd)D_{\la1,\cdots,n\ra}^{\frac{1}{p}}}^2_p
	& = \Big(\lambda \norm{(a+\mu^2_n d)D_{\la1,\cdots,n-1\ra}^{\frac{1}{p}}}^p_p \\ \nonumber & \qquad \qquad
	+ (1-\lambda)\norm{(a-\mu^{-2}_n d)D_{\la1,\cdots,n-1\ra}^{\frac{1}{p}}}^p_p\Big)^{\frac{2}{p}}\\
	& \ge \norm{aD_{\la1,\cdots,n-1\ra}^{\frac{1}{p}}}^2_p
	+ C(\mu)(p-1)\norm{dD_{\la1,\cdots,n-1\ra}^{\frac{1}{p}}}^2_p, \nonumber
	\end{align}
	where $C(\mu)$ is the constant in Lemma \ref{lem-modified-CL}.
\end{lem}
\begin{proof}
It is a direct application of Proposition \ref{form} and Lemma 
\ref{lem-modified-CL}, if one notices that 
for any $a, d \in \Gamma_{\la 1,...,n-1\ra}$,
	$\Phi(a + y_n d) = \left[\begin{array}{cc} a + \mu^2_n d & 0\\ 0& a - \mu^{-2}_n d \end{array}\right].$
\end{proof}

	\begin{lem}\label{lem-gamma-estimate}
	Let $b,c\in \Gamma_{\la 1,\cdots,n-1\ra}$. Then, we have
		$$\norm{\gamma_n b D^{\frac{1}{p}}_{\la 1,\cdots,n \ra}}_p
		\ge \lambda^{\frac{1}{p}} (\mu^2_n + \mu^{-2}_n)^{\frac{1}{2}}
		\norm{b D^{\frac{1}{p}}_{\la 1,\cdots,n-1 \ra}}_p$$
	and
		$$\norm{\gamma^*_n c D^{\frac{1}{p}}_{\la 1,\cdots,n \ra}}_p
		\ge (1-\lambda)^{\frac{1}{p}} (\mu^2_n + \mu^{-2}_n)^{\frac{1}{2}}
		\norm{c D^{\frac{1}{p}}_{\la 1,\cdots,n-1 \ra}}_p.$$
	\end{lem}
\begin{proof}
By \eqref{eq-estimate-I} we get
	\begin{align*}
	\norm{\gamma_n b D^{\frac{1}{p}}_{\la 1,\cdots,n \ra}}_p
	& \ge \frac{1}{\norm{\gamma^*_n}_\infty}
	\norm{\gamma^*_n\gamma_n b D^{\frac{1}{p}}_{\la 1,\cdots,n \ra}}_p\\
	& = \frac{1}{\norm{\gamma^*_n}_\infty}
	\norm{ (\mu^{-2}_nb + y_nb )D^{\frac{1}{p}}_{\la 1,\cdots,n \ra}}_p\\
	& \geq \frac{1}{\norm{\gamma^*_n}_\infty}
	\lambda^{\frac{1}{p}} (\mu^2_n + \mu^{-2}_n) \norm{b D^{\frac{1}{p}}_{\la 1,\cdots,n-1 \ra}}_p\\
	& = \lambda^{\frac{1}{p}} (\mu^2_n + \mu^{-2}_n)^{\frac{1}{2}}
	\norm{b D^{\frac{1}{p}}_{\la 1,\cdots,n-1 \ra}}_p.
	\end{align*}
Note that the equality in the last line holds by the fact
	$$\norm{\gamma_n}_\infty = \norm{\gamma^*_n}_\infty = \sqrt{\mu^2_n + \mu^{-2}_n},$$
which is a direct application of Proposition \ref{form}.

The estimate for $\norm{\gamma^*_n c D^{\frac{1}{p}}_{\la 1,\cdots,n
\ra}}_p$ is similar.
\end{proof}

\begin{proof}[Proof of Theorem \ref{thm-eps}]
We follow the idea of Carlen/Lieb and Biane to use the induction on
$n$, where $I = \{\pm 1, \cdots, \pm n\}$.  We assume that we have the
conclusion for $n-1$ and consider the case $n$.  Every element in
$\Gamma_{\la 1,\cdots,n \ra}$ can be uniquely expressed as
	$$X = a+\gamma_nb+\gamma^*_nc+y_nd,$$
where $a,b,c,d \in \Gamma_{\la 1,\cdots,n-1 \ra}$.
Note that
$\{D_{\la n\ra}^{\frac{1}{2}},\, \gamma_n D_{\la n\ra}^{\frac{1}{2}},\, \gamma^*_n D_{\la n\ra}^{\frac{1}{2}},\, y_n D_{\la n\ra}^{\frac{1}{2}}\}$
is an orthogonal set in $L^2(\Gamma_{\la n \ra})$ with
	$$\norm{D_{\la n\ra}^{\frac{1}{2}}}_2=\norm{y_n D_{\la n\ra}^{\frac{1}{2}}}_2=1,\;
	\norm{\gamma_n D_{\la n\ra}^{\frac{1}{2}}}_2 = \mu^{-1}_n\;\; \text{and}\;\;
	\norm{\gamma^*_n D_{\la n\ra}^{\frac{1}{2}}}_2 = \mu_n.$$
For example, $\norm{y_n D_{\la n\ra}^{\frac{1}{2}}}^2_2
= \text{tr}_{\Gamma_{\la n \ra}}(D_{\la n\ra}^{\frac{1}{2}} y^2_n D_{\la n\ra}^{\frac{1}{2}})
= \tau^\eps(y^2_n) = \la y_n 1, y_n 1 \ra = 1$.
Moreover, $y_n1 = x_{-n}x_n$ so that $P^\eps_t(y_n) = e^{-2t}y_n$.
Thus, by applying (3) of Proposition \ref{prop-gamma}, we get that the four terms in $X$ are orthogonal and 
	\begin{align}\label{eq-2-norm}
	\lefteqn{\norm{P^\eps_t(X)D_{\la1,\cdots,n\ra}^{\frac{1}{2}}}^2_2}\\
	&= \norm{P^\eps_t(a)D_{\la1,\cdots,n-1\ra}^{\frac{1}{2}}}^2_2
	+\mu^{-2}_n e^{-2t}\norm{P^\eps_t(b)D_{\la1,\cdots,n-1\ra}^{\frac{1}{2}}}^2_2\nonumber \\
	&\;\;\;\;+\mu^2_n e^{-2t}\norm{P^\eps_t(c)D_{\la1,\cdots,n-1\ra}^{\frac{1}{2}}}^2_2
	+e^{-4t}\norm{P^\eps_t(d)D_{\la1,\cdots,n-1\ra}^{\frac{1}{2}}}^2_2.\nonumber
	\end{align}
Now we estimate $\norm{X}_p$.
Since the map replacing $\gamma_n$ into $-\gamma_n$ is a $\tau^\eps$-preserving $*$-isomorphism of $\Gamma_{\la1,\cdots,n\ra}$, Proposition \ref{prop-Lp-ext} implies that
	$$\norm{(a+\gamma_nb+\gamma^*_nc+y_nd)D_{\la1,\cdots,n\ra}^{\frac{1}{p}}}_p
	= \norm{(a-\gamma_nb-\gamma^*_nb+y_nd)D_{\la1,\cdots,n\ra}^{\frac{1}{p}}}_p.$$
By the optimal convexity inequality (\cite{BCL} or \cite{CL93}) we have
	\begin{align}\label{eq-p-norm}
	\norm{XD_{\la1,\cdots,n\ra}^{\frac{1}{p}}}^2_p
	& \ge \norm{(a+y_nd)D_{\la1,\cdots,n\ra}^{\frac{1}{p}}}^2_p
	+ (p-1)\norm{(\gamma_nb+\gamma^*_nc)D_{\la1,\cdots,n\ra}^{\frac{1}{p}}}^2_p\\
	& = I + (p-1)II. \nonumber
	\end{align}

The estimate for $I$ is Lemma \ref{lem-a+yn}. For $II$, note that $\gamma_n bD_{\la1,\cdots,n\ra}^{\frac{1}{p}}$ and
$\gamma^*_n cD_{\la1,\cdots,n\ra}^{\frac{1}{p}}$ have disjoint support.
Indeed, we have
	$$(\gamma_n bD_{\la1,\cdots,n\ra}^{\frac{2}{p}} b^*\gamma^*_n)
	(\gamma^*_n c D_{\la1,\cdots,n\ra}^{\frac{2}{p}} c^*\gamma_n)=0$$
and
	\begin{align*}
	\lefteqn{(D_{\la1,\cdots,n\ra}^{\frac{1}{p}}b^*\gamma^*_n \gamma_n bD_{\la1,\cdots,n\ra}^{\frac{1}{p}})
	(D_{\la1,\cdots,n\ra}^{\frac{1}{p}}c^*\gamma_n \gamma^*_n cD_{\la1,\cdots,n\ra}^{\frac{1}{p}})}\\
	& = D_{\la1,\cdots,n\ra}^{\frac{1}{p}}b^*bD_{\la1,\cdots,n\ra}^{\frac{1}{p}}
	\gamma^*_n\gamma_n \gamma_n\gamma^*_n D_{\la1,\cdots,n\ra}^{\frac{1}{p}}c^*c D_{\la1,\cdots,n\ra}^{\frac{1}{p}}\\
	& = 0
	\end{align*}
by (3) of Proposition \ref{prop-gamma}.
Thus, by orthogonality and Lemma \ref{lem-gamma-estimate} we have
	\begin{align}\label{eq-estimate-II}
	II
	& = \Big( \norm{\gamma_n b D_{\la1,\cdots,n\ra}^{\frac{1}{p}}}^p_p
	+ \norm{\gamma^*_n c D_{\la1,\cdots,n\ra}^{\frac{1}{p}}}^p_p \Big)^{\frac{2}{p}}\\
	& \ge \norm{\gamma_n b D_{\la1,\cdots,n\ra}^{\frac{1}{p}}}^2_p + \norm{\gamma^*_n c D_{\la1,\cdots,n\ra}^{\frac{1}{p}}}^2_p \nonumber \\
	& \ge \lambda^{\frac{2}{p}} (\mu^2_n + \mu^{-2}_n)
	\norm{b D^{\frac{1}{p}}_{\la 1,\cdots,n-1 \ra}}^2_p
	+ (1-\lambda)^{\frac{2}{p}} (\mu^2_n + \mu^{-2}_n)
	\norm{c D^{\frac{1}{p}}_{\la 1,\cdots,n-1 \ra}}^2_p. \nonumber
	\end{align}
By combining \eqref{eq-2-norm}, \eqref{eq-p-norm}, \eqref{eq-estimate-I} and \eqref{eq-estimate-II} we get
	$$\norm{P^\eps_t(X)D_{\la1,\cdots,n\ra}^{\frac{1}{2}}}^2_2
	\le \norm{XD_{\la1,\cdots,n\ra}^{\frac{1}{p}}}^2_p$$
provided that
	$$e^{-2t} \le \min \{ (\mu^4_n+1)^{1-\frac{2}{p}}(p-1),\; \sqrt{C(\mu_n)}\sqrt{p-1}\},$$
where $C(\mu_n)$ is the constant in Lemma \ref{lem-modified-CL} for $\mu = \mu_n$.
\end{proof}

\section{Approximation by central limit procedure}\label{sec-CentralLimit}

The aim of this section is to use a standard approximation procedure
to go from the baby Fock model to the $q$-Araki-Woods algebras.
Most of the arguments are easy adaptations of \cite{N06}, so we will simply
sketch them. 

In section \ref{sec-babyFock} we constructed generalized
baby gaussians $\gamma_i$ associated with the parameters $\mu_i$ for
$1\le i\le n$ by starting with the index set $I = \{\pm 1, \pm 2,
\cdots, \pm n\}$.  In this section we apply the same construction
using the increased index set
	$$\widetilde{I} = \{ (i,j) : 1\le i\le n,\; 1\le j \le m \} \cup \{ (-i,-j) : 1\le i\le n,\; 1\le j \le m \},$$
so that we can get generalized baby gaussians $\gamma_{i,j}$ associated with the parameter $\mu_i$ for $1\le i \le n$, $1\le j\le m$
and the von Neumann algebra $\Gamma_{n,m}$ generated by $\{ \gamma_{i,j} : 1\le i \le n,\; 1\le j\le m\}$.
Note that the ``choice of sign" function $\eps$ in this case would be
	$$\eps : \widetilde{I} \times \widetilde{I} \rightarrow \{\pm 1\}$$
satisfying
	$$\begin{array}{l}
	\eps((i_1,i_2), (j_1,j_2)) = \eps((j_1,j_2), (i_1,i_2)),\\
	\eps((i_1,i_2), (i_1,i_2))=-1,\\
	\eps((i_1,i_2), (j_1,j_2)) = \eps((\abs{i_1},\abs{i_2}), (\abs{j_1},\abs{j_2})),
	\; \forall (i_1,i_2), (j_1,j_2)\in \widetilde{I}. \end{array}$$
Now we replace $\eps((i_1,i_2), (j_1,j_2))$, $(i_1,i_2) \prec (j_1,j_2)\in \widetilde{I}$ with a family of i.i.d. random variables
with
	$$P(\eps((i_1,i_2), (j_1,j_2))= -1) = \frac{1-q}{2}, \;\; P(\eps((i_1,i_2), (j_1,j_2))= 1) = \frac{1+q}{2},$$
where $(i_1,i_2) \prec (j_1,j_2)$ means $i_1<j_1$ or $i_1=j_1, i_2<j_2$.
We set
	$$s_{i,m} = \frac{1}{\sqrt{m}}\sum^m_{j=1}\gamma_{i,j}.$$
Then, the Speicher's central limit procedure (\cite{Sp92, N06}) tells us the following.
	\begin{prop}\label{prop-n-var}
	For any $*$-polynomial $Q$ in $n$ non-commuting variables we have
		$$\lim_{m\rightarrow \infty}\tau^\eps(Q(s_{1,m}, \cdots, s_{n,m})) = \tau_q(Q(g_{q,1},\cdots, g_{q,n}))$$
	for almost every $\eps$.
	\end{prop}
Since the set of all non-commuting $*$-polynomials is countable, we can find a choice of sign $\eps$ such that the above is true for any $Q$.
In the sequel we fix such an $\eps$.

Now we would like to transfer this convergence in distribution into $L^p$-norm convergence
using Nou's ultraproduct approach (\cite[Theorem 4.3, Section 5.2]{N06}.
If we set $g_{i,m} = \text{Re}(s_{i,m})$, $g_{-i,m} = \text{Im}(s_{i,m})$ and
$G_i = \text{Re}(g_{q,i})$, $G_{-i} = \text{Im}(g_{q,i})$, $1\le i \le n$,
then by Proposition \ref{prop-n-var} for any polynomial $P$ in $2n$ non-commuting variables we have
	\begin{equation}
	\lim_{m\rightarrow \infty}\tau^\eps(P(g_{-n,m}, \cdots, g_{n,m})) = \tau_q(P(G_{-n},\cdots, G_n)).
	\end{equation}
We need to truncate $g_{j,m}$ to get a uniform control on the operator
norms.  Let $C>0$ be a constant satisfying $\norm{G_j}_{\Gamma_q} < C$
for any $\abs{j} \le n$.  We consider the function $h$ on $\Real$ with
$h(x) = 1_{(-C,C)}(x)x,\; x\in \Real$ and set $\tilde{g}_{i,m} =
h(g_{i,m}),\; 1\le i \le n.$ From \cite[Lemma 5.7]{N06} and the
discussion after it, we have
	\begin{prop}\label{prop-state-preserving-map}
	Let $\U$ be a fixed free ultrafilter on $\n$, $(\A, \tau) = \Pi_{m,\U}(\Gamma_{n,m}, \tau^\eps)$, and $p \in \A$ be the support of $\tau$.
	Then we have the following normal state-preserving $*$-isomorphism.
		$$\Theta : (\Gamma_q, \tau_q) \rightarrow (\A, \tau),\;\;
		P(G_{-n},\cdots, G_n) \mapsto p\cdot( P(\tilde{g}_{-n,m}, \cdots, \tilde{g}_{n,m}) )_{m,\U}\cdot p,$$
	where $P$ is any polynomial in $2n$ non-commuting variables.
	\end{prop}

Then by Proposition \ref{prop-state-preserving-map}, Proposition \ref{prop-Lp-ext} and \cite[Theorem 3.6]{Ray}
for any polynomial $P$ in $2n$ non-commuting variables we have
	$$\lim_{m,\U}\norm{P(\tilde{g}_{-n,m}, \cdots, \tilde{g}_{n,m})D_m^{\frac{1}{p}}}_p = \norm{P(G_{-n},\cdots, G_n)D_q^{\frac{1}{p}}}_p,$$
where $D_m$ is the density of $\tau^\eps$ restricted to $\Gamma_{n,m}$.
Now we need to replace $\tilde{g}_{i,m}$ back with $g_{i,m}$.

	\begin{lem}\label{lem-Lp-conv}
	Let $\U$ be a fixed free ultrafilter on $\n$ and $1\le p\le 2$. For any polynomial $P$ in $2n$ non-commuting variables we have
		$$\lim_{m,\U}\norm{P(g_{-n,m}, \cdots, g_{n,m})D_m^{\frac{1}{p}}}_p
		= \norm{P(G_{-n},\cdots, G_n)D_q^{\frac{1}{p}}}_p.$$
	\end{lem}
\begin{proof}
In the proof of \cite[Lemma 5.7]{N06} it is shown that
	$$\lim_{m\rightarrow \infty}\abs{\tau^\eps(\tilde{g}_{n,j_1}\cdots \tilde{g}_{n,j_{k-1}}(g_{n,j_k} - \tilde{g}_{n,j_k})
	g_{n,j_{k+1}} \cdots g_{n,j_l})} = 0$$
for any indices $j_1,\cdots, j_l$ and $1\le k\le l$.
By taking involution inside the functional $\tau^\eps$ we also get
	$$\lim_{m\rightarrow \infty}\abs{\tau^\eps(g_{n,j_1}\cdots g_{n,j_{k-1}}(g_{n,j_k} - \tilde{g}_{n,j_k})
	\tilde{g}_{n,j_{k+1}} \cdots \tilde{g}_{n,j_l})} = 0.$$
If we apply the above limits repeatedly, then we have
	\begin{equation}\label{lim1}
	\lim_{m\rightarrow \infty}\abs{\tau^\eps(\tilde{g}_{n,j_1}\cdots \tilde{g}_{n,j_{k-1}}g_{n,j_k}\cdots g_{n,j_l})
	-\tau^\eps(g_{n,j_1}\cdots g_{n,j_{k-1}}g_{n,j_k}\cdots g_{n,j_l})} = 0
	\end{equation}
and
	\begin{equation}\label{lim2}
	\lim_{m\rightarrow \infty}\abs{\tau^\eps(g_{n,j_1}\cdots g_{n,j_{k-1}}\tilde{g}_{n,j_k}\cdots \tilde{g}_{n,j_l})
	-\tau^\eps(\tilde{g}_{n,j_1}\cdots \tilde{g}_{n,j_{k-1}}\tilde{g}_{n,j_k}\cdots \tilde{g}_{n,j_l})} = 0.
	\end{equation}
Now we consider any polynomial $P$ in $2n$ non-commuting variables, then we have
	\begin{align*}
	\lefteqn{\norm{P(g_{-n,m}, \cdots, g_{n,m})D_m^{\frac{1}{2}} - P(\tilde{g}_{-n,m}, \cdots, \tilde{g}_{n,m})D_m^{\frac{1}{2}}}^2_2}\\
	& = \abs{\tau^\eps(P^*P - \tilde{P}^*P - P^*\tilde{P} + \tilde{P}^*\tilde{P})}
	\le \abs{\tau^\eps(P^*P - \tilde{P}^*P)} + \abs{\tau^\eps(P^*\tilde{P} - \tilde{P}^*\tilde{P})},
	\end{align*}
where $P$ and $\tilde{P}$ denote $P(g_{-n,m}, \cdots, g_{n,m})$ and $P(\tilde{g}_{-n,m}, \cdots, \tilde{g}_{n,m})$, respectively.
Since $P^*P - \tilde{P}^*P$ and $P^*\tilde{P} - \tilde{P}^*\tilde{P}$ are linear combinations of the terms of the forms
	$$\tilde{g}_{n,j_1}\cdots \tilde{g}_{n,j_{k-1}}g_{n,j_k}\cdots g_{n,j_l} - g_{n,j_1}\cdots g_{n,j_{k-1}}g_{n,j_k}\cdots g_{n,j_l}$$
and
	$$g_{n,j_1}\cdots g_{n,j_{k-1}}\tilde{g}_{n,j_k}\cdots \tilde{g}_{n,j_l}
	- \tilde{g}_{n,j_1}\cdots \tilde{g}_{n,j_{k-1}}\tilde{g}_{n,j_k}\cdots \tilde{g}_{n,j_l},$$
respectively, \eqref{lim1} and \eqref{lim2} imply that
	$$\lim_{m\rightarrow \infty}
	\norm{P(g_{-n,m}, \cdots, g_{n,m})D_m^{\frac{1}{2}} - P(\tilde{g}_{-n,m}, \cdots, \tilde{g}_{n,m})D_m^{\frac{1}{2}}}_2 = 0.$$
Since $L^2(\tau^\eps)$ embeds into $L^p(\tau^\eps)$ contractively we get
	$$\lim_{m\rightarrow \infty}
	\norm{P(g_{-n,m}, \cdots, g_{n,m})D_m^{\frac{1}{p}} - P(\tilde{g}_{-n,m}, \cdots, \tilde{g}_{n,m})D_m^{\frac{1}{p}}}_p = 0,$$
so that
	\begin{align*}
	\lim_{m,\U}\norm{P(g_{-n,m}, \cdots, g_{n,m})D_m^{\frac{1}{p}}}_p
	& = \lim_{m,\U}\norm{P(\tilde{g}_{-n,m}, \cdots, \tilde{g}_{n,m})D_m^{\frac{1}{p}}}_p\\
	& = \norm{P(G_{-n},\cdots, G_n)D_q^{\frac{1}{p}}}_p.
	\end{align*}
\end{proof}

	\begin{rem}{\rm We can extend Lemma \ref{lem-Lp-conv} for the case $2<p<\infty$.
	}
	\end{rem}

The following lemma is a non-tracial version of \cite[Lemma 5]{Bi97}.

	\begin{lem}\label{lem-lim-OU-semigp}
	For any $*$-polynomial $Q$ in $n$ non-commuting variables and $1\le p \le 2$ we have
		$$\lim_{m\rightarrow \infty}\norm{P^\eps_t(Q(s_{1,m}, \cdots, s_{n,m}))D_m^{\frac{1}{p}}}_p
		= \norm{P^q_t(Q(g_{q,1},\cdots, g_{q,n}))D_q^{\frac{1}{p}}}_p.$$
	\end{lem}
\begin{proof}
The proof is essentially the same as \cite[Lemma 5]{Bi97}, so that we omit it.
Note that we need Lemma \ref{lem-Lp-conv} for the conclusion.
\end{proof}
\begin{proof}[Proof of Theorem \ref{thm-main-if}] By a standard density argument it is enough to consider the case dim$\Hi = n$.
Then for any $*$-polynomial $Q$ in $n$ non-commuting variables and a fixed free ultrafilter $\U$ on $\n$ we have
	$$\norm{P^q_t(Q(g_{q,1},\cdots, g_{q,n}))D_q^{\frac{1}{2}}}_2
	= \lim_{m, \U}\norm{P^\eps_t(Q(s_{1,m}, \cdots, s_{n,m}))D_m^{\frac{1}{2}}}_2$$
by Lemma \ref{lem-lim-OU-semigp}. Theorem \ref{thm-eps} implies that
	$$\norm{P^\eps_t(Q(s_{1,m}, \cdots, s_{n,m}))D_m^{\frac{1}{2}}}_2 \le \norm{Q(s_{1,m}, \cdots, s_{n,m}))D_m^{\frac{1}{p}}}_p$$
if $e^{-2t} \le C\alpha_\mu^{4-\frac{8}{p}}(p-1)$, where $C$ is the constant in Theorem \ref{thm-eps}.
Applying Lemma \ref{lem-Lp-conv} we get
	$$\norm{P^q_t(Q(g_{q,1},\cdots, g_{q,n}))D_q^{\frac{1}{2}}}_2
	\le \norm{P^q_t(Q(g_{q,1},\cdots, g_{q,n}))D_q^{\frac{1}{p}}}_p.$$
\end{proof}

\begin{rem}
{\rm
For the most general case of $\Gamma_q(H_\Real,(U_t))$ we use the discretization argument in \cite[section 6]{N06},
where the following embedding has been established.

For a fixed free ultrafilter $\U$ on $\n$ we have the following normal state-preserving $*$-isomorphism.
	$$\Theta : (\Gamma_q(H_\Real,(U_t)), \tau_q) \rightarrow \Pi_{n,\U}(\Gamma_n, \tau_n),\;\;
	G(e_i) \mapsto p\cdot( G_n(e_i) )_{n,\U}\cdot p,$$
where $(\Gamma_n, \tau_n)$'s are almost periodic $q$-Araki-Woods algebras, $G(e_i)$, $G_n(e_i)$'s are corresponding gaussians
and $p \in \Pi_{n,\U}\Gamma_n$ be the support of $\Pi_{n,\U}\tau_n$.

Then, the same ultraproduct argument as above proves Theorem \ref{thm-general}.
}
\end{rem}

\section{1-dimensional estimate}\label{sec-1dim}
We consider the ``only if" direction by examining 1-dimensional behavior as usual.
We start by an estimate of the $L^p$-norm of $g_i$, the $q$-gaussian with the parameter $\mu_i$.
As $g_i^*g_i$ is in the centralizer of $\varphi$.
$$\|g_iD^{\frac 1 p}_{\la i \ra}\|_p=\|D^{\frac 1 p}_{\la i
\ra}g_i^*g_iD^{\frac 1 p}_{\la i
\ra}\|_{p/2}^{\frac 1 2}=\varphi((g_i^*g_i)^{p/2})^{\frac 1 p}.$$ The self-adjoint
element $y=g_i^*g_i$ can be seen as a commutative random variable in
some probability space with measure induced by $\varphi$. It is well
known that $q$-creations are bounded for $-1\leq q<1$ (\cite[Lemma 4]{BS91}) so $\|y\|_\infty\sim \mu_i^2$ with constants
depending only on $q$. Moreover, we have already
seen that $\|y\|_1=\frac 1 {\mu_i^2}$ and $\|y\|_2\sim 1$. It follows
from the H\"older inequality, that
$\|y\|_p=\varphi((g_i^*g_i)^{p})^{\frac 1 p}\sim \mu_i^{2-4/p}$ with constants
depending only on $q$. So we conclude that for $p\geq 2$ 
$$\|g_iD^{\frac 1 p}_{\la i \ra}\|_p\sim \mu_i^{1-\frac 4p}.$$
By duality  if $P^q_t$ can be extended to a contraction from $L^p(\Gamma_q)$ into $L^2(\Gamma_q)$,
then it can also be extended from $L^2(\Gamma_q)$ into $L^{p'}(\Gamma_q)$, where $\frac{1}{p} + \frac{1}{p'} = 1$, so
	$$e^{-t}\leq \mu_i^{2-\frac 4p}.$$
That is
	\begin{thm}\label{thm-unbdd-case}
	Suppose that $\alpha_\mu = \sup_n \mu_n = \infty$,
	then $P^q_t$ can not be extended to a contraction from $L^p(\Gamma_q)$ into $L^2(\Gamma_q)$ for any $1\leq p<2$.
	\end{thm}

We give a more precise estimate for $p\to 1$. Let $n\in \n$, and we set
$a(\eps) = (1+\eps g_i)D^{\frac{1}{2n}}_{\la i \ra}$, $\eps>0$.  Then
we have
	$$\abs{a(\eps)}^{2n}
	= (D^{\frac{1}{2n}}_{\la i \ra}(1+\eps g_i+\eps g^*_i+\eps^2 g^*_ig_i)D^{\frac{1}{2n}}_{\la i \ra})^n.$$
If we expand the right hand side, then we get
	\begin{align*}
	D_{\la i \ra}
	& + \eps (D^{\frac{1}{2n}}_{\la i \ra} g_i D^{\frac{1}{2n}}_{\la i \ra} D^{\frac{n-1}{n}}_{\la i \ra}
	+ \cdots + D^{\frac{n-1}{n}}_{\la i \ra} D^{\frac{1}{2n}}_{\la i \ra} g_i D^{\frac{1}{2n}}_{\la i \ra})\\
	& + \eps (D^{\frac{1}{2n}}_{\la i \ra} g^*_i D^{\frac{1}{2n}}_{\la i \ra} D^{\frac{n-1}{n}}_{\la i \ra}
	+ \cdots + D^{\frac{n-1}{n}}_{\la i \ra} D^{\frac{1}{2n}}_{\la i \ra} g^*_i D^{\frac{1}{2n}}_{\la i \ra})\\
	& + n\eps^2 D_{\la i \ra}g^*_ig_i\\
	& + \eps^2 (D^{\frac{1}{2n}}_{\la i \ra} g^*_i D^{\frac{1}{n}}_{\la i \ra}
	g_i D^{\frac{1}{2n}}_{\la i \ra} D^{\frac{n-2}{n}}_{\la i \ra} + \cdots \\
	& \;\;\;\;\;\;\;\; + D^{\frac{1}{2n}}_{\la i \ra} g_i D^{\frac{1}{n}}_{\la i \ra}
	g^*_i D^{\frac{1}{2n}}_{\la i \ra} D^{\frac{n-2}{n}}_{\la i \ra} + \cdots \\
	& \;\;\;\;\;\;\;\; + D^{\frac{1}{2n}}_{\la i \ra} g^*_i D^{\frac{2}{n}}_{\la i \ra}
	g_i D^{\frac{1}{2n}}_{\la i \ra} D^{\frac{n-3}{n}}_{\la i \ra} + \cdots \\
	& \;\;\;\;\;\;\;\; + D^{\frac{1}{2n}}_{\la i \ra} g_i D^{\frac{2}{n}}_{\la i \ra}
	g^*_i D^{\frac{1}{2n}}_{\la i \ra} D^{\frac{n-3}{n}}_{\la i \ra} + \cdots)\\
	& + o(\eps^2)D_{\la i \ra}.
	\end{align*}
Thus, we have
	\begin{align*}
	\text{tr}(\abs{a(\eps)}^{2n})
	& = 1 + n\eps^2\mu^{-2}_i
	+ \eps^2 \Big(\sum^{n-1}_{k=1}(n-k)\mu^{\frac{4}{n}k-2}_i
	+ \sum^{n-1}_{k=1}k\mu^{-\frac{4}{n}(n-k)+2}_i\Big) + o(\eps^2)\\
	& = 1 + n \eps^2 \sum^{n-1}_{k=0}\mu^{\frac{4}{n}k-2}_i + o(\eps^2)
	= 1 + n \eps^2 \mu^{-2}_i \frac{\mu^4_i - 1}{\mu^{\frac{4}{n}}_i - 1} + o(\eps^2),
	\end{align*}
so that
	\begin{align*}
	\norm{P^t_q(a(\eps))}_{2n} & = \norm{a(e^{-t} \eps)}_{2n}\\
	& = \Big(1 + n e^{-2t} \eps^2 \mu^{-2}_i \frac{\mu^4_i - 1}{\mu^{\frac{4}{n}}_i - 1} + o(\eps^2)\Big)^{\frac{1}{2n}}\\
	& = 1 +  \frac{e^{-2t}\eps^2}{2}\mu^{-2}_i \frac{\mu^4_i - 1}{\mu^{\frac{4}{n}}_i - 1} + o(\eps^2)\\
	& \ge 1 + \frac{n}{2} e^{-2t}\eps^2 \mu^{2-\frac{4}{n}}_i + o(\eps^2).
	\end{align*}
	
Consequently, $\norm{P^t_q(a(\eps))}_{2n} \le \norm{a(\eps)}_2$ implies that
	$$1 + \frac{n}{2} e^{-2t}\eps^2 \mu^{2-\frac{4}{n}}_i + o(\eps^2) 
	\le 1 +  \frac{\eps^2}{2}\mu^{-2}_i + o(\eps^2),$$
which means
	\begin{equation}\label{1-dim-estimate}
	e^{-2t} \le \frac{1}{n}\mu^{-4+\frac{4}{n}}_i \le 2\mu^{-4+\frac{8}{2n}}_i\frac{1}{2n-1}
	\end{equation}
by taking $\eps \rightarrow 0$.
By duality we get the following.

	\begin{thm}\label{thm-1-dim}
	Let $\frac{1}{p} = 1 - \frac{1}{2n}$, $n(\ge 2) \in \n$. Then $\norm{P^t_q}_{L^p \rightarrow L^2} \le 1$ implies that
		$$e^{-2t} \le 2\alpha_\mu^{4-\frac{8}{p}}(p-1).$$
	\end{thm}

If we turn back to the baby Fock model, this one dimensional estimate
can be extended for all $1<p<2$. That is, Theorem \ref{thm-eps} is optimal.

\begin{prop}
Let $d\in \mathbb M_n$ be an invertible self-adjoint matrix, and $g\in \mathbb
M_n$, such that $dg=\lambda gd$ for some $\lambda>1$. Then, for
any $p>2$,
$$\| (1+\eps g)d \|_p^p={\rm Tr}\, d^{p} + \eps^2\left(\Big(\frac p
2 +c_{p,\lambda}\Big) \, {\rm Tr}\, d^{p}g^*g + 
c_{p,\frac 1 \lambda}{\rm Tr}\, d^{p}gg^*\right) + O(\eps^2), $$
where 
$$c_{p,\lambda}=\frac {{\lambda^p}-1} {(\lambda^2-1)(1-\frac
1{\lambda^2})}-\frac {p\lambda^2}
{2({\lambda^2}-1)}.$$
\end{prop}

\begin{proof}
We use the well known fact that on positive definite matrices,
the map $f: x\mapsto x^{p/2}$ is $C^\infty$. Moreover, its derivative
at $x$ can be expressed easily in terms of the spectral decomposition of
$x$ and divided differences of $f$; if $x=\sum s p_s$ is the spectral
decomposition of $x$, then for $h\in \mathbb M_n^{sa}$:
$${\rm diff}_x f. h= \sum_{s,t} f_1(s,t) p_shp_t$$
$${\rm diff^2}_x f. (h,h)= 2.\sum_{s,t,u} f_2(s,t,u) p_shp_thp_u$$
where 
$$f_1(a,b)=\left\{ \begin{array}{ll}\frac{f(a)-f(b)}{a-b} & \textrm{if } a\neq b \\
f'(a) & \textrm{if } a= b \end{array}\right.$$
$$f_2(a,b,c)=\left\{ \begin{array}{ll}
\frac{f_1(a,c)-f_1(b,c)}{a-b} & \textrm{if } a\neq b \\
\lim_{h\to 0}\frac{f_1(a+h,c)-f_1(a,c)}{h} & \textrm{if } a= b \end{array}\right.$$
Under the trace, for our choice of $f$ :
$${\rm Tr}\,({\rm diff}_x  f. h)= \frac p 2 \, {\rm Tr}\, x^{p/2-1} h$$ 
$${\rm Tr}\,({\rm diff^2}_x  f. (h,h))=  2 \cdot {\rm Tr}\, \Big(\sum_{s,t} f_2(s,t,s) p_shp_th\Big).$$ 
We want the expansion at the second order in $\eps$ of 
$$\|(1+\eps g)d\|_p^p= {\rm Tr}\, (d^2 + \eps d(g+g^*)d +\eps^2 d
g^*gd)^{p/2}.$$ By the above formula, with $x=d^2$, the first order
term is $\frac p 2 \, {\rm Tr}\, d^{p}(g+g^*)=0$ because of the
commutation relation as $\lambda\neq 1$.

By the Taylor expansion, the second order term has two contributions,
one from the first derivative, the other coming from one half the second one.
The first is given by $\frac p 2 \, {\rm Tr}\, d^{p}g^*g$. The second
is more involved; let $d=\sum_{\alpha\in \sigma(d)} \alpha p_{\alpha}$
be its spectral decomposition, we get
\begin{eqnarray*}
 A&=&{\rm Tr}\, \Big(\sum_{\alpha,\beta
\in \sigma (d)} f_2(\alpha^2,\beta^2,\alpha^2) 
p_\alpha d(g+g^*)dp_\beta d(g+g^*)d \Big)\\ 
&=&{\rm Tr}\, \Big(\sum_{\alpha,\beta
\in \sigma (d)} (\alpha\beta)^2f_2(\alpha^2,\beta^2,\alpha^2) 
p_\alpha (g+g^*)p_\beta (g+g^*) p_\alpha\Big).
\end{eqnarray*}
The relation $dg=\lambda gd$ gives that $P(d) g=g P(\lambda d)$ for
any polynomial $P$. It yields $p_\alpha g= g p_{\frac \alpha\lambda}$, where
$p_{\frac\alpha\lambda}$ is zero if $\frac\alpha \lambda$ is not a eigenvalue of $d$.  
In particular, 
$$p_\alpha (g+g^*)p_\beta(g+g^*) p_\alpha= g^2p_{\frac{\alpha}{\lambda^2}}
p_{\frac\beta\lambda}p_\alpha+gg^*p_\alpha p_{ \beta\lambda}p_\alpha+g^*g p_\alpha
p_{ \frac\beta\lambda}p_\alpha +g^{*2}p_{\alpha \lambda^2}p_\alpha p_{ \beta\lambda} p_\alpha.$$
Thus
$$
 A={\rm Tr}\Big(\sum_{\alpha\in \sigma (d)}
\frac {\alpha^4}{\lambda^2} f_2(\alpha^2,\frac{\alpha^2}{\lambda^2},\alpha^2)
p_\alpha gg^*\Big)+ {\rm Tr}\Big(\sum_{\alpha\in \sigma (d)}
{\alpha^4}{\lambda^2} f_2(\alpha^2,{\alpha^2}{\lambda^2},\alpha^2)
p_\alpha g^*g\Big).$$
Then, 
$$\frac {\alpha^4}{\lambda^2}
f_2(\alpha^2,\frac{\alpha^2}{\lambda^2},\alpha^2) =\alpha^p
\left(\frac p
{2(\lambda^2-1)}-\frac {1-\frac 1{\lambda^p}} {(\lambda^2-1)(1-\frac
1{\lambda^2})}\right)=\alpha^p c_{p,\frac 1 \lambda}.$$
Finally, 
$$A= c_{p,\frac 1 \lambda}{\rm Tr}\, d^p gg^* +c_{p, \lambda}{\rm Tr}\, d^p g^*g.$$
\end{proof}

To conclude, using the notation of Section 3, we apply it with 
$d=D_{\la 1,...,n\ra}^{\frac 1 p}$, $g=\gamma_n$, $\lambda=\mu_n^{\frac 4 p}$.
Recall that ${\rm Tr}\, d^{p}g^*g=\mu_n^{-2}$ and  ${\rm Tr}\, d^{p}gg^*=\mu_n^{2}$,
\begin{eqnarray*}
\|(1+\eps \gamma_n)D_{\la 1,...,n\ra}^{\frac 1 p}\|_p^p&=&1+\eps^2 \left(
\frac p{2\mu_n^2} - \frac {p \mu_n^{\frac 8 p}}{2\mu_n^2(\mu_n^{\frac 8 p}-1)}+
\frac {p\mu_n^2} {2 (\mu_n^{\frac 8 p}-1)}\right)+ O(\eps^2)\\
& =&  1 + \eps^2 \cdot \frac{p}{2\mu_n^2} \cdot \frac {\mu_n^4 -1}{\mu_n^{\frac 8p}-1}+ O(\eps^2).
\end{eqnarray*}
Then, the conclusion about the optimality follows as above.
\bibliographystyle{amsplain}
\providecommand{\bysame}{\leavevmode\hbox
to3em{\hrulefill}\thinspace}

\end{document}